\documentclass[12pt,oneside,reqno]{amsart}

\usepackage{mathtools}
\usepackage{amsthm} 
\usepackage{amsrefs}
\usepackage{amssymb}
\usepackage{amsmath}
\usepackage{amsfonts}
\usepackage{graphicx}									 
\usepackage{stmaryrd}                  
\usepackage{footmisc}                  
\usepackage{paralist}
\usepackage{wrapfig}
\usepackage[draft]{pdfcomment}
\usepackage{bbm}
\usepackage[utf8]{inputenc}
\usepackage{lmodern}
\usepackage[ngerman,british]{babel}
\usepackage[arrow, matrix, curve]{xy}

\usepackage{capt-of}
\usepackage{enumitem}
\usepackage{tikz-cd,xkeyval,ifthen,calc}
\usepackage{graphicx}
\usepackage{multicol}
\usepackage{fancyvrb}
\usetikzlibrary{babel}
\usepackage{todonotes}
\usepackage{hyperref}
\usepackage{aliascnt}
\author{Nina Lebedeva}
\address{\tiny{Saint Petersburg State University, 7/9 Universitetskaya nab., St. Petersburg, 199034, Russia}}
\address{\tiny{St. Petersburg Department of V.A. Steklov Institute of Mathematics of the Russian Academy of Sciences, 27 Fontanka nab., St. Petersburg, 191023, Russia}}
\email{lebed@pdmi.ras.ru}

\author{Artem Nepechiy}
\address{\scriptsize{Institut für Mathematik, Universität Augsburg, Universitätsstraße 14, 86159 Augsburg, Germany}}
\email{artem.nepechiy@math.uni-augsburg.de}

\makeatletter
\@namedef{subjclassname@2020}{%
	\textup{2020} Mathematics Subject Classification}
\makeatother
\subjclass[2020]{53C30, 57S05, 58D05}
\keywords{Riemannian manifolds of low regularity, locally homogeneous spaces}
\title{Locally homogeneous $C^0$-Riemannian manifolds}

\newtheorem{defi}{Definition}[section]

\newaliascnt{thm}{defi}
\newtheorem{thm}[thm]{Theorem}
\aliascntresetthe{thm}

\newaliascnt{prop}{defi}
\newtheorem{prop}[prop]{Proposition}
\aliascntresetthe{prop}

\newaliascnt{cor}{defi}
\newtheorem{cor}[cor]{Corollary}
\aliascntresetthe{cor}

\newaliascnt{lem}{defi}
\newtheorem{lem}[lem]{Lemma}
\aliascntresetthe{lem}

\DeclareMathOperator{\Iso}{\operatorname{Iso}}
\newtheorem*{prop*}{Proposition}
\newtheorem*{Main Theorem}{Main Theorem}

\newtheorem*{Corollary}{Corollary}
\begin{document}
	
	\begin{abstract}
		We show that locally homogeneous $\mathcal{C}^0$-Riemannian manifolds are smooth.
	\end{abstract}
	\maketitle
	\section{Introduction}
	In this paper we prove that if a $\mathcal{C}^0$-Riemannian manifold is locally homogeneous, then it is indeed smooth, more precisely we obtain the following theorem:
	\begin{Main Theorem}[Local homogeneity implies smoothness]
		Let $(M,g_0)$ be a locally homogeneous $\mathcal{C}^0$-Riemannian manifold and denote by $d_{g_0}$ the induced metric, then $(M,d_{g_0})$ is isometric to a smooth Riemannian manifold.
	\end{Main Theorem}
	In fact we show that for any point there is a small neighborhood $U$, such that the set of local isometries on $U$, which will be denoted by $U_G$, forms a local Lie group with Lie algebra $\mathfrak{g}$ acting transitively on $U$. The isotropy local isometries determine a local Lie group $U_H$ with Lie algebra  $\mathfrak{h}$ and $U$ is isometric to the coset space $U_G/U_H$ carrying an invariant metric with respect to the left action of $U_G$ (for definitions see \cites{spiro1993remark,spiro1992lie,mostow1950extensibility,MR1088511}).\par
	In particular all spaces appearing in the main theorem are determined by Lie algebras $\mathfrak{g} \supset \mathfrak{h}$ together with a scalar product $\langle \cdot , \cdot \rangle$ on $\mathfrak{g}/\mathfrak{h}$, which is skew symmetric with respect to the adjoint action of $\mathfrak{h}$ on $\mathfrak{g}/\mathfrak{h}$ \cite{spiro1993remark}. Thus, they are given by purely algebraic data. Moreover, this implies that $M$ and its Riemannian metric are real analytic. \par

	Our result in some sense generalizes the Meyers-Steenrod theorems \cite{MR1503467},
	which in particular asserts that the isometry group of a smooth Riemannian manifold is a Lie group.
	Several theorems are known in that direction: Metric spaces with geometric assumptions such as curvature conditions imply regularity of the isometry group.
	For example isometry groups of Alexandrov spaces or RCD*(K,N) spaces are known to be Lie groups \cite{MR1273464} \cite{MR3914958}.

	The question "When is a homogeneous/locally homogeneous space a smooth manifold?" has been investigated in \cites{MR997464,pediconi2019compactness,Pediconi_2020}.

	In \cite{MR997464}[Theorem 7] Berestovskii studied when a globally homogeneous inner metric space is isometric to a homogeneous Riemannian manifold. His findings show
	in particular that a homogeneous Alexandrov space is in fact a smooth Riemannian manifold. In contrast to that we obtain  a theorem of a local nature.
	One can show (using \cites{MR2262730,MR3925112} for upper curvature bounds and \cite{MR1274133} for lower bounds)
	that a locally homogeneous space
	with an upper or lower curvature bound in the sense of Alexandrov is a
	$\mathcal{C}^0$-Riemannian manifold. Hence our main theorem implies:
	\begin{Corollary}
		Let $X$ be a locally homogeneous, locally compact, length space of finite Hausdorff dimension. If there exists a point together with a convex neighbourhood admitting a curvature bound from either above or below in the sense of Alexandrov, then $X$ is isometric to a smooth Riemannian manifold.
	\end{Corollary}
	It would be interesting to obtain a full description of locally homogeneous, locally compact length 
	spaces similar to \cite{MR997464} without assuming any regularity on the metric and topology. This could be considered as a metric version of the Bing-Borsuk conjecture \cite{MR2488667}.              
	
	There exist different results in the local setting, however they are making stronger assumptions on the regularity of the manifold.
	In \cite{MR131248}
	Singer showed: If a complete, simply connected Riemannian manifold is curvature
	homogeneous and the derivatives of the curvature tensor agree up to some order 
	at all points, then the manifold is globally homogeneous.
	If the Riemannian metric is complete and sufficiently smooth, the conclusion of our main theorem follows from this result. While the proof is essentially local and completeness does not play a central role, it relies heavily on the existence of high order derivatives of the metric\cites{MR131248,MR1088511}.\par
	Lately local versions with lower regularity have been obtained by Pediconi \cites{pediconi2019compactness,Pediconi_2020} with different additional assumptions on the space and the group action.\par
	Riemannian manifolds with low regularity do not satisfy classical results in Riemannian geometry: There is no meaningful notion of curvature and shortest curves do not need to solve a differential equation, they may branch and the injectivity radius may be zero \cite{MR38111}. Shortest curves do not even need to be $\mathcal{C}^1$ \cite{MR38111}. We refer to \cite{MR3358543} for some basic properties of $\mathcal{C}^0$-Riemannian manifolds and to \cites{MR283727, MR1109664,MR1318157} for further results. \par
	
	A metric space $M$ is called locally homogeneous if the pseudogroup of local isometries acts transitively on $M$.
	One important problem and the difference to the non-local case (as considered by Berestovskii) is that the pseudogroup  of local isometries is a priori not known to be a local group. The technical tool to overcome this obstacle is to extend local isometries, defined on arbitrary small balls, to balls of fixed radius. Once the pseudogroup is established to be a local group, one can apply structure theory of locally compact groups \cites{MR2680491,MR1464908} to show that it is a local Lie group. We then construct a local isometry between our metric space $M$ and a local quotient of the local group equipped with an invariant Riemannian metric.\par

	The paper is organized as follows: In \autoref{sec: Preliminiaries} we fix notation, explain what a $C^0$-Riemannian manifold is and give definitions and notions regarding local groups. In \autoref{sec:Ismetry extensions} we prove that every local isometry can be extended to an isometry of fixed size.
	In \autoref{subseq: local lie group structure} we explain how to obtain a local topological group and prove that some restriction is a local Lie group. After that we will explain how to obtain a left-invariant metric on the quotient, which is isometric to some open subset of $M$.
	
	\textbf{Acknowledgements.} The main theorem was motivated by research conducted in \cite{bohm2019optimal}, our solution allows for simplifications therein. We are grateful to Christoph Böhm for introducing us (via Alexander Lytchak) to the problem. We would like to warmly thank Alexander Lytchak for valuable comments, discussions and his support. \par

	The authors were partially supported by the DFG grant SPP 2026.
	\section{Preliminaries}\label{sec: Preliminiaries}
	\subsection{\texorpdfstring{$C^0$}{C0}-Riemannian manifolds}
	In this subsection we collect all definitions and results regarding $C^0$-Riemannian manifolds.
	\begin{defi}[$C^0$-Riemannian manifold]
		A $C^0$-Riemannian manifold is a pair $(M,g_0)$ consisting of a $\mathcal{C}^1$-manifold $M$ together with a continuous Riemannian metric $g_0$.
	\end{defi}
	The Riemannian metric $g_0$ induces a canonical length structure, which in turn induces an intrinsic metric $d_{g_0}$ on $M$. This allows us to formulate local homogeneity in purely metric terms. We denote open (closed) balls with radius $r$ around the point $x$ by $B_r(x)$ ($\overline{B}_r(x)$). 
	\begin{defi}[Local homogeneity]
		A metric space $M$ is called locally homogeneous if for every $x,y \in M$ there exists $r>0$ and
		an isometry $f:B_r(x) \rightarrow B_r(y)$ satisfying $f(x)=y$.
	\end{defi}
	We want to make frequent use of the upcoming lemma, which is implied by the $C^0$-Riemannian manifold structure.
	\begin{lem}[Maps are bilipshitz]\label{lem: Maps are bilip}
		Let $(M,g_0)$ be a $C^0$-Riemannian manifold then the coordinate maps are bilipshitz.
	\end{lem}
	
	\begin{proof}
		Compare \cite{MR2216252}[Section 3.2].
	\end{proof}

	\subsection{Local topological groups}
	In this subsection we introduce, for the convenience of the reader, the basic definitions and notations regarding local groups. The exposition is mostly taken from \cite{MR2680491}.
	\begin{defi}[Local Group]\label{defi: local group}
		A local topological group $G=(G,\Omega,e, m, i)$ is a Hausdorff topological space $G$ together with a neutral element $e \in G$, a partially defined but continuous multiplication operation $ m : \Omega \rightarrow G$ for some open domain $\Omega \subset G \times G$, and a partially defined but continuous inversion operation $i: G \rightarrow G$ obeying the following axioms:
		\begin{enumerate}
			\item $\Omega$ is an open neighborhood of $G \times \lbrace 1 \rbrace \cup \lbrace 1 \rbrace \times G$.
			\item If $g,h,k \in G$ satisfy $m(g,h) , m(h,k) \in \Omega$ and $m(m(g , h) , k),m(m(g, m(h , k)) \in \Omega$ then  $m(m(g, h) , k)= m(g, m(h ,k))$.
			\item For all $g \in G$ one has $m(g, e) =g = m(e, g)$.
			\item If $g \in G$, then $m(g,i(g))=e=m(i(g),g)$.
		\end{enumerate}
	\end{defi}
	We will use the shorthand notation $g\cdot h$ for $m(g,h)$. We call an open neighborhood $U$ of $e \in G$ symmetric if it satisfies $U=i(U)$. Note that if $U$ is an arbitrary open neighborhood of $e$, then $U \cap i^{-1}(U)$ is an open symmetric neighborhood of $e$. If the local group $G$ is additionally a smooth manifold and the local group operations are smooth, then  we say that $G$ is a local Lie group. The basic example of a local group is the restriction of a topological group.
	
	\begin{defi}[Restriction of a local group]\label{def restriction of group}
		Let $G$ be a local topological group and $U$ a symmetric open neighborhood of the
		identity of $G$. We have a local group $G \vert_ U$, it has the subspace $U$ as underlying
		space, $e_G$ as its neutral element, the restriction of inversion to $U$ as its inversion,
		and the restriction of the product to
		$$ \Omega_U:= \lbrace (x,y) \in \Omega \cap (U \times U) : m(x,y) \in U \rbrace $$
		as its product. Such a local group $G \vert_ U$ is called a restriction of $G$.
	\end{defi}
	
	We want to define the notion when local topological groups are equivalent. For this we need the definition of local isomorphism.
	
	\begin{defi}[Locally isomorphic top. groups]
		Let $G=(G,\Omega, e,m,i)$ and $G'=(G',\Omega', e',m',i')$ be local topological groups. A morphism from $G$ to $G'$ is a continuous function $f:G \rightarrow G'$ such that
		\begin{enumerate}
			\item $f(e)=e'$ and $(f\times f) (\Omega) \subset \Omega'$. 
			\item $f(i(g)) = i'(f(g))$ for all $g \in G$.
			\item $f(m(g,h))= m'(f(g),f(h))$ for all $(g,h) \in \Omega$.
		\end{enumerate}
		
		We say $G$ and $G'$ are locally isomorphic if there exist open symmetric neighborhoods $U$ and $U'$ of $e$ and $e'$ in $G$ and $G'$ respectively, $f:U \rightarrow U'$ a homeomorphism and $f:G\vert_U \rightarrow G'\vert_{U'}, f^{-1}: G'\vert _ {U'} \rightarrow G \vert_U$ are morphisms.
	\end{defi}
	\section{Isometry Extensions}\label{sec:Ismetry extensions}
	The goal of this section is to obtain an extension property. That is local isometries defined on arbitrary balls can be extended to balls of fixed radius. More precisely we want to prove the following statement:
	\begin{prop}[Extension Property]\label{Prop: Extension Property} Let $(M,g_0)$ be a locally homogeneous $C^0$-Riemannian manifold. Then there exists an open ball $B_{r_0}(x_0) \subset M$ and $R>0$ such that for all points $x,y\in B_{r_0}(x_0)$ and all $r<R$ any isometry $ f: B_{r}(x)\to B_{r}(y)$ satisfying $f(x)=y$ can be extended uniquely to an isometry $F: B_{R}(x)\to B_{R}(y)$.
	\end{prop}
	Once this Proposition is established, we can define a local group structure on the set of local isometries. This will be carried out in \autoref{subseq: local lie group structure}.\par
	Moreover, a detailed analysis of our proof gives the following quantitative estimate on $R$: If $\rho>0$ satisfies the condition that every loop in $B_\rho(x)$ is contractible in $B_{3\rho}(x)$ and the closure $\overline{B}_{3\rho}(x)$ is compact then we have $R \geq \rho$.\par 
	Now we turn to the proof of \autoref{Prop: Extension Property}, it is subdivided into two steps executed in \autoref{small isometries} and \autoref{large extension} respectively.\par
	The first subsection deals with small extensions of isometries, meaning that given an isometry $f:B_r(x) \rightarrow B_r(y)$ it can be extended to an isometry $F:B_{r+\varepsilon}(x) \rightarrow B_{r+\varepsilon}(y)$, where $\varepsilon \ll r$. The main ingredient is the Lipschitz version of the Hilbert smith conjecture, which shows that the isometry groups, we are dealing with, are actually Lie groups. This enables us to formulate statements about extensions of isometries in terms of Lie groups.  \par
	The second subsection deals with large extensions of isometries, meaning that given an isometry $f:B_r(x) \rightarrow B_r(y)$ it can be extended to an isometry $F:B_{R}(x) \rightarrow B_{R}(y)$ where $r \ll R$. The main ingredient is to extend the isometry along paths using the results obtained in \autoref{small isometries}, thus proving \autoref{Prop: Extension Property}.

	\subsection{Small Isometry Extension}\label{small isometries}

	The goal of this subsection is to obtain:
	\begin{prop}[Existence of local Extension]\label{prop: Existence of local extensions III}
		Let $(M,g_0)$ be a locally homogeneous $C^0$-Riemannian manifold. Then there exists an open ball $B_{r_0}(x_0) \subset M$ and $R>0$ such that for all but countably many $r<R$ there is $\varepsilon_r>0$ such that every pointed isometry $ f: B_r(x) \rightarrow B_r(y)$ can be extended to an isometry $F:B_{r+\varepsilon_r}(x) \rightarrow B_{r+\varepsilon_r}(y)$ for every $x,y\in  B_{r_0}(x_0)$.
	\end{prop}

	The first step is to obtain a well-behaved subset of $M$ on which we will define the local group of isometries. Observe that for arbitrary $x,y \in M$ and the pointed isometry $f:B_r(x) \rightarrow B_r(y)$, coming from the local homogeneity condition, there can in general be no lower bound for $r>0$. Using the Baire category theorem we want to find $B_{r_0}(x_0)$ such that for all $x,y \in B_{r_0}(x_0)$ one has a lower bound on $r>0$.
	\begin{lem}[Lower domain bound]\label{Lower domain bound} Denote by $(M,g_0)$ a locally homogeneous $C^0$-Riemannian manifold. Then there exists an open ball $B_{r_0}(x_0)$ and $R>0$ such that for all points $x,y\in B_{r_0}(x_0)$ there exists an isometry $f_{xy}: B_{R}(x)\to B_{R}(y)$ satisfying $f(x)=y$.
	\end{lem}
	\begin{proof}
		Fix some point $x_0 \in M$ and consider a closed compact ball $\overline{B}$ containing $x_0$, then by local homogeneity for every $x \in \overline{B}$ there exists maximal $r_x>0$ and an isometry $f_x:B_{r_x}(x_0) \rightarrow B_{r_x}(x)$ satisfying $f(x_0)=x$. Define for $n \in \mathbb{N}$ the set
		$$ \mathcal{F}_{\frac{1}{n}} = \left \lbrace x \in \overline{B}  :  r_x \geq \frac{1}{n} \right  \rbrace. $$
		By the above one has $\overline{B}= \bigcup_{n \in \mathbb{N}} \mathcal{F}_{\frac{1}{n}}$ and each $\mathcal{F}_{\frac{1}{n}}$ is closed. 
		Therefore Baire's category theorem \cite{MR3728284}[Theorem 48.2] implies that for some $m \in \mathbb{N}$ the set $\mathcal{F}_{\frac{1}{m}}$ has non-empty interior. Thus, there is a point $x_0$ and a radius $r_0>0$ such that $B_{r_0}(x_0) \subset \mathcal{F}_{\frac{1}{m}}$. Now $f_{xy} := f_y \circ f_x^{-1}$ yields the desired map.
	\end{proof}
	Since the Hilbert-Smith theorem will be used frequently, we recall it for the convenience of the reader.
	\begin{thm}[Hilbert-Smith theorem \cite{MR1464908}]\label{Thm: Hilbert smith}
		If $G$ is a locally compact group, which acts effectively by Lipschitz homeomorphisms on a Riemannian manifold, then $G$ is a Lie group.
	\end{thm}
	\begin{lem}[Pointed isometries form a Lie group]\label{lem: Pointed isometries are Lie group}
		In the situation of \autoref{Prop: Extension Property} the group of pointed isometries
		$$ \operatorname{Iso}_x(B_r(x)) := \lbrace f:B_r(x) \rightarrow B_r(x) : f \text{ isometry, }  f(x)=x \rbrace $$
		is a compact Lie group.
	\end{lem}
	\begin{proof}
		By \cite{MR1393941}[Corollary 4.8] we have that $\operatorname{Iso}_x(B_R(x))$ is compact. Moreover, the isometry group $\operatorname{Iso}_x(B_R(x))$ of $B_R(x)$ acts effectively by Lipschitz homeomorphisms on an open set of $\mathbb{R}^n$, therefore $\operatorname{Iso}_x(B_R(x))$ is a a Lie group by the Hilbert-Smith theorem (\autoref{Thm: Hilbert smith}).
	\end{proof}
	
	\begin{lem}[Uniqueness of Extensions]\label{lem: Extensions are unique}
		Let $(M,g_0)$ be a $C^0$-Riemannian manifold, $r>0$ and $f:B_r(x) \rightarrow B_r(y)$ an isometry satisfying $f(x)=y$. If there exists $\varepsilon>0$ and extensions $F,G:B_{r+\varepsilon}(x) \rightarrow B_{r+\varepsilon}(y)$ of $f$, then one has $F=G$.
	\end{lem}
	\begin{proof}
		Consider the isometry
		$$H: B_{r+\varepsilon}(x) \rightarrow B_{r+\varepsilon}(x), z \mapsto G^{-1} \circ F (z).$$
		The group $ \overline{\langle H \rangle} \subset \operatorname{Iso}_x(B_{r+\varepsilon}(x))$ generated by $H$ is a compact Lie Group, since it is a closed subgroup of a compact Lie group by \autoref{lem: Pointed isometries are Lie group}. Observe that all elements of $\overline{\langle H \rangle}$ fix the open set $B_r(x)$. This is a contradiction to the Newmann Theorem \cite{MR0413144}[Theorem 9.5], since the fixed point set of a compact Lie group cannot contain an open set.\par
	\end{proof}
	\begin{lem}[Local extensions of isometry groups]\label{lem: Existence of local extensions I}
		Denote by $B_{r_0}(x_0)$ the set coming from \autoref{Lower domain bound}, then for every $x \in B_{r_0}(x_0)$ and every sufficiently small $r>0$ there exists $R>0$ such that for all $\varepsilon_1< \varepsilon_2<R$ one has
		$$\operatorname{Iso}_x(B_{r- \varepsilon_1}(x))=\operatorname{Iso}_x(B_{r-\varepsilon_2}(x)) .$$
	\end{lem}
	\begin{proof}
		By \autoref{lem: Extensions are unique} we have a natural inclusion $\operatorname{Iso}_x(B_{r+\varepsilon}(x)) \hookrightarrow \operatorname{Iso}_x(B_r(x)) $ just by restricting the maps to the smaller domain. Thus $L_i:=\operatorname{Iso}_x(B_{r- \frac{1}{i}}(x))$ defines by \autoref{lem: Pointed isometries are Lie group} a sequence of compact Lie groups satisfying $L_{i+1} \subset L_{i}$. Hence the $L_i$ must stabilize, meaning there exists $N \in \mathbb{N}$ such that $L_n=L_m$ for all $n,m \geq N$. This proves statement.\par
	\end{proof}
	\begin{cor}[Local extensions of isometry groups are homogeneous]\label{cor: Existence of local extensions II}
		Denote by $B_{r_0}(x_0)$ the set coming from \autoref{Lower domain bound}, then there is $R>0$ such that for all but countably many $R>r>0$ there exists $\varepsilon(r)>0$ such that
		$$\operatorname{Iso}_x(B_{r}(x))=\operatorname{Iso}_x(B_{r+\varepsilon(r)}(x))$$
		for all $x \in B_{r_0}(x_0)$.
	\end{cor}
	\begin{proof}
		Fix $x \in B_{r_0}(x)$. We show that for all but countably many $r>0$ an isometry can be extended, i.e $\operatorname{Iso}_x(B_{r}(x))=\operatorname{Iso}_x(B_{r+\varepsilon(r)}(x))$ for some $\varepsilon(r)>0$. This immediately follows from \autoref{lem: Existence of local extensions I}. Indeed without loss of generality assume $R=1$. We want to show that $$C:=\lbrace r \in (0,1): \text{ for all } \varepsilon>0 : \operatorname{Iso}_x(B_{r}(x))\neq \operatorname{Iso}_x(B_{r+\varepsilon(r)}(x)) \rbrace$$ is countable. By \autoref{lem: Existence of local extensions I} we have 
		$$ \alpha:= \inf \lbrace \tau \in [0,1] : C \cap [\tau,1] \text{ is countable } \rbrace< \infty,$$
		since the set appearing in the definition is not empty. If we can show $\alpha=0$, then the statement follows. Assume $\alpha>0$, then by \autoref{lem: Existence of local extensions I} $C \cap [\alpha-\varepsilon, 1]$ is countable for some $\varepsilon>0$. This is a contradiction to the choice of $\alpha$ and the statement follows.\par 
		The second statement is that this $r>0$ does not depend on the point, i.e.
		$$\operatorname{Iso}_x(B_{r}(x))=\operatorname{Iso}_x(B_{r+\varepsilon(r)}(x)) \Rightarrow \operatorname{Iso}_y(B_{r}(y))=\operatorname{Iso}_y(B_{r+\varepsilon(r)}(y))$$
		for all $x,y \in B_{r_0}(x_0)$.\par
		Fix an element $g \in \operatorname{Iso}_y(B_{r}(y))$ and consider a pointed isometry $f:B_{R}(x) \rightarrow B_R(y)$ provided by \autoref{Lower domain bound}. If $r\leq R$, then one has 
		$$f^{-1} \circ g \circ f \vert_{B_r(x)} \in \operatorname{Iso}_x(B_{r}(x)).$$
		By assumption there exists an isometry
		$$\overline{f^{-1} \circ g \circ f}:B_{r+\varepsilon(r)}(x) \rightarrow B_{r+\varepsilon(r)}(x) $$
		extending $f^{-1} \circ g \circ f \vert_{B_r(x)}$. Now $f \circ \overline{f^{-1} \circ g \circ f} \circ f^{-1}$ is an extension of $g$ and by the uniqueness result \autoref{lem: Extensions are unique}. The statement follows.
	\end{proof}
	\begin{lem}[Existence of local extensions]\label{lem: Existence of local extensions III}
		Denote by $B_{r_0}(x_0)$ the set coming from \autoref{Lower domain bound}, then there exists $R>0$ such that for all but countably many $r<R$ there is $\varepsilon_r>0$ such that every pointed isometry $ f: B_r(x) \rightarrow B_r(y)$ can be extended to an isometry $F:B_{r+\varepsilon_r}(x) \rightarrow B_{r+\varepsilon_r}(y)$ for every $x,y\in  B_{r_0}(x_0)$.
	\end{lem}
	\begin{proof}
		Choose $r>0$ according to \autoref{cor: Existence of local extensions II} and an isometry $G:B_R(x) \rightarrow B_R(y)$ provided by \autoref{Lower domain bound}, where $x$ denotes a point in $B_{r_0}(x_0)$. We can assume without loss of generality $r+\varepsilon(r) <R$. Then $G^{-1} \circ f$ is an element of $\Iso_x(B_r(x))$ and thus has an extension
		$$\overline{G^{-1} \circ f}: B_{r+\varepsilon(r)}(x) \rightarrow  B_{r+\varepsilon(r)}(x).$$
		Now $F:= G \circ \overline{G^{-1} \circ f}$ is the desired extension of $f$.
	\end{proof}

	\subsection{Large extension}\label{large extension}
	The goal of this subsection is to prove \autoref{Prop: Extension Property} using \autoref{lem: Existence of local extensions III}. Consider an isometry $f:B_r(x) \rightarrow B_r(y)$, a point $z \in B_r(x)$ and a point $z' \notin B_r(x)$ The idea is to extend $f$ along a path $\gamma$ from $z$ to $z'$ by repeatedly applying \autoref{lem: Existence of local extensions III}. It remains to prove that this extension procedure is well-defined. This will take up most of the subsection. Once this is established, we will be able to construct a \glqq large\grqq{} extension.\par
	
	\begin{defi}[Isometry caterpillar]
		Let $(M,g_0)$ be a $C^0$-Riemannian manifold. Consider a path $\gamma:[a,b] \rightarrow M$
		parametrized proportional to arc-length by constant speed $C$ (i.e. $L(\gamma\vert_{[t_1,t_{2}]})= C(t_2-t_1)$ for all $t_1,t_2 \in [a,b]$ with $t_1 < t_2$). An $r$-isometry caterpillar along $\gamma$ is a family of isometries $f_t:B_r(\gamma(t))\rightarrow B_r(f(\gamma(t)))$ for $t \in [a,b]$, such that for all $t_1,t_2 \in [a,b]$ with $\vert t_1 -t_2 \vert < \frac{r}{10C}$ we have
		$$ f_{t_1} \vert_{B_r(\gamma(t_1)) \cap B_r(\gamma(t_2))} = f_{t_2} \vert_{B_r(\gamma(t_1)) \cap B_r(\gamma(t_2))}.  $$
		We say an isometry caterpillar is fat, if every isometry $f_t:B_r(\gamma(t))\rightarrow B_r(f(\gamma(t)))$ can be extended to an isometry $F_t:B_{10r}(\gamma(t)) \rightarrow B_{10r}(f(\gamma(t)))$.
	\end{defi}
	Observe that the above definition is invariant under linear reparameterizations.
	The next Lemma shows: If two $r$-isometry caterpillars (for the same path $\gamma$) agree at the starting point, then they agree everywhere. 
	\begin{lem}[Caterpillar uniqueness]\label{lem: Caterpillar uniqueness}
		Let $(M,g_0)$ be a $C^0$-Riemannian manifold and $\gamma:[a,b] \rightarrow M$ a path parametrized proportional to arc-length. If $f_t^1,f_t^2$ are two $r$-isometry caterpillars along $\gamma$ satisfying $f_a^1 \equiv f_a^2$ on $B_r(\gamma(a))$, then one has $f_t^1 = f_t^2$ on $B_r(\gamma(t))$ for all $t \in [a,b]$.
	\end{lem}
	\begin{proof}
		Parameterize $\gamma$ by arc-length and consider $t \in [a,b]$ such that $\vert a-t \vert \leq r/2$. Then $f_t^1$ and $f_t^2$ agree on $B_{r/2}(\gamma(t))$ by triangle inequality. Now \autoref{lem: Extensions are unique} implies $f_t^1 \equiv f_t^2$ on $B_r(\gamma(t))$. Thus the claim follows inductively by subdividing $\gamma$ into pieces of length smaller that $r/2$. 
	\end{proof}
	The upcoming Lemma proves that fat caterpillars can be concatenated, observe that this statement fails without the fatness condition.
	\begin{lem}[Concatenation of caterpillars]\label{concatenation of caterpillars}
		Let $(M,g_0)$ be a $C^0$-Riemannian manifold, $\gamma_1:[a,b] \rightarrow M, \gamma_2:[b,c] \rightarrow M$ paths parametrized by arc-length with $\gamma_1(b)=\gamma_2(b)$ and $f_t^1, f_t^2$ fat $r$-isometry caterpillars along $\gamma_1$ and $\gamma_2$ respectively satisfying $f_b^1 \equiv f_b^2$ on $B_r(\gamma_1(b))=B_r(\gamma_2(b))$.\par
		Then the family of isometries $f_t:B_r(\gamma(t))\rightarrow B_r(f(\gamma(t)))$ for $t \in [a,c]$ is a fat isometry caterpillar along the concatenation of $\gamma_1$ and $\gamma_2$.
	\end{lem}
	\begin{proof}
		We need to check for $ b - r/10 < t_1 < b < t_2 < b +  r/10$ the condition
		$$ f_{t_1} \vert_{B_r(\gamma_1(t_1)) \cap B_r(\gamma_2(t_2))} = f_{t_2} \vert_{B_r(\gamma_1(t_1)) \cap B_r(\gamma_2(t_2))}.  $$
		Consider a point $z \in B_r(\gamma_1(t_1)) \cap B_r(\gamma_2(t_2))$ and the extensions 
		$$F_{t_i}^i : B_{10r}(\gamma_i(t_i)) \rightarrow B_{10r}(f_{t_i}^i(\gamma_i(t_i))) \text{ for } i =1,2$$
		and
		$$F_{b}^1 : B_{10r}(\gamma_1(b)) \rightarrow B_{10r}(f_{t_1}^1(\gamma_1(b)))$$ 
		provided by the fatness condition of the caterpillar. By construction $F_{t_2}^2$ and $F^1_b$ agree in a small neighborhood of $\gamma(t_2)$, therefore they have also to satisfy $F^2_{t_2}(z)=F^1_b(z)$ by \autoref{lem: Extensions are unique}, otherwise we would obtain two different extensions fixing an open neighborhood of $\gamma(t_2)$. A similar argument gives $F_{t_1}^1(z)=F_b^1(z)$. Since these maps are just extensions of $f_{t_1},f_{t_2}$, we have $f_{t_1}(z)=f_{t_2}(z)$. Hence the result.
	\end{proof}
	\begin{lem}[Existence of caterpillar isometries]\label{lem: Existence of caterpillar isometries}
		Let $(M,g_0)$ be a $C^0$-Riemannian manifold. There exists an open, simply connected set $U$ and $R>0$ such that for any $x,y \in U$, $0<r<R$, every pointed isometry $f:B_r(x)\rightarrow B_r(y)$, every $z \in B_R(x)$ and any rectifiable path $\gamma: [a,b] \rightarrow M$ from $x$ to $z$ there is a $\rho>0$ and a fat $\rho$-isometry caterpillar extending $f\vert_{B_{\rho}(x)}$.
	\end{lem}
	\begin{proof}
		In view of \autoref{lem: Existence of local extensions III} we can take $U:=B_{\frac{r_0}{10}}(x_0)$ (where $x_0, r_0$ are coming from \autoref{Lower domain bound}) and find $r',\varepsilon'>0$ such that for all $x,y \in U$ every pointed isometry $g:B_{r'}(x)\rightarrow B_{r'}(y)$ can be extended to an isometry $G:B_{r'+\varepsilon'}(x) \rightarrow B_{r'+\varepsilon'}(y)$. \par
		Set $R:=r_0/10$ and consider the isometry $f:B_r(x)\rightarrow B_r(y)$ from the assumptions of \autoref{lem: Existence of caterpillar isometries} and a rectifiable path $\gamma:[a,b]\rightarrow M$ from $x$ to $z \in B_R(x)$ parameterized by arc-length. By \autoref{lem: Existence of local extensions III} we can choose $r' \ll r$.  Subdivide $\gamma$ into paths $\gamma_1, \ldots , \gamma_N$ of length less than $\varepsilon'/10$. Inductively we can construct an isometry caterpillar along $\gamma$. For the induction start set $f_t \vert_{B_{r'/10}(\gamma(t))}:= F\vert_{B_{r'/10}(\gamma(t))}$ for $t \in [a, a+\varepsilon'/10]$, where $F$ denotes the extension $F:B_{r'+\varepsilon'}(x) \rightarrow B_{r'+\varepsilon'}(y)$ of $f:B_{r'}(x) \rightarrow B_{r'}(y)$ (this is possible since $B_{r'/10}(\gamma(t)) \subset B_{r'/10+\varepsilon'}(x)$). \par
		By definition the compatibility condition for $f_t$ is satisfied and thus it is a fat isometry caterpillar for some $\rho>0$. Now by \autoref{concatenation of caterpillars} we can extend this construction to the concatenation of $\gamma_1$ and $\gamma_2$. This way we obtain the result. 
	\end{proof}
	\begin{lem}[Close path endpoint compatibility]\label{lem:Close path endpoint compatibilty}
		Let $(M,g_0)$ be a $C^0$-Riemannian manifold and $\gamma_1,\gamma_2:[a,b]\rightarrow M$ two paths parametrized proportional to arc-length satisfying $0,9 \leq L(\gamma_1)/ L(\gamma_2) \leq 1.1$. Moreover, assume there is an $r>0$ such that $$d(\gamma_1(t),\gamma_2(t)) < r/2  \text{ for all } t \in [a,b]$$
		and $\gamma_1,\gamma_2$ admitt fat $r$-isometry caterpillars $f_t^1$ and $f_t^2$ with the property
		$$ f_a^1\vert_{B_{r}(\gamma_1(a))\cap B_r(\gamma_2(a))}= f_a^2\vert_{B_{r}(\gamma_1(a))\cap B_r(\gamma_2(a))}.  $$
		Then we have 
		$$f_b^1\vert_{B_{r}(\gamma_1(b))\cap B_r(\gamma_2(b))}= f_b^2\vert_{B_{r}(\gamma_1(b))\cap B_r(\gamma_2(b))}.  $$
	\end{lem}
	\begin{proof}
		After reparametrization we can assume that $\gamma_1$ is parametrized by arc-length and $\gamma_2$ is parametrized by constant speed $C \in [1,1.1]$. Find a subdivision $a=t_1, \ldots, t_N=b $ of $[a,b]$ such that the length of the curves $\gamma_1\vert_{[t_i,t_{i+1}]}, \gamma_2\vert_{[t_i,t_{i+1}]}$ is $\leq \frac{r}{1.1 \cdot 10}$ for $i=1,\ldots, N-1$. We will show the stronger claim
		$$ f_t^1\vert_{B_{r}(\gamma_1(t))\cap B_r(\gamma_2(t))}= f_t^2\vert_{B_{r}(\gamma_1(t))\cap B_r(\gamma_2(t))}  \text{ for all } t \in[a,b]. $$
		For $t \in [t_1, t_2]$ consider the extensions $F_a^1,F_t^1,F_t^2$. All these isometries agree in a neighborhood of $\gamma(a)$: $F_a^1$ and $F_t^1$ by the caterpillar condition and $F_a^1,F_t^2$ by assumption together with the caterpillar condition. By the triangle inequality and the fatness condition all these extensions are defined on $B_{r}(\gamma_1(t))\cap B_r(\gamma_2(t))\neq \emptyset$ and they agree by \autoref{lem: Extensions are unique}. In particular the restrictions agree as well. Thus, the claim is shown for $\gamma_1\vert_{[t_1,t_{2}]}, \gamma_2\vert_{[t_1,t_{2}]}$. Using the same argument one gets by induction the statement for all of $\gamma_1,\gamma_2$.
	\end{proof}
	\begin{lem}[Existence of global extensions]\label{lem:Existence of global extension}
		Let $(M,g_0)$ be a $C^0$-Riemannian manifold. Then there exists an open $U \subset M$ and $R>0$ such that for all points $x,y\in U$ and all $r<R$ any isometry $ f: B_{r}(x)\to B_{r}(y)$ satisfying $f(x)=y$ can be extended to a local isometry $F: B_{R}(x)\to B_{R}(y)$.
	\end{lem}
	\begin{proof}
		Construction of $F$:~\\  
		Denote by $U$ the open set and by $R$ the bound coming from \autoref{lem: Existence of caterpillar isometries}. Fix the pointed isometry $f:B_r(x) \rightarrow B_r(y)$, a point $z \in B_R(x)$ and a rectifiable path $\gamma:[a,b]\rightarrow M$ from $x$ to $z$ (for example one randomly chosen shortest path from $x$ to $z$). Then by \autoref{lem: Existence of caterpillar isometries} and \autoref{lem: Caterpillar uniqueness} there exists a unique fat isometry caterpillar $f_t$, which coincides with $f$ in a neighborhood of the point $x$. Define $F(z):= f_b(\gamma(b))$.\par 
		It remains to show that $F$ is well-defined, i.e. the definition does not depend on the path $\gamma$. Let $\gamma'$ be another rectifiable path from $x$ to $z$. Since the argument is local, by \autoref{lem: Maps are bilip} we know that there is a homotopy $\gamma_t$ between $\gamma$ and $\gamma'$ keeping the endpoints fixed such that $L(\gamma_t)$ depend continuously on the parameter $t$. We can find $0=t_1,\ldots , t_N = 1$ such that $\gamma_{t_i}$ and $\gamma_{t_{i+1}}$ satisfy the assumption of \autoref{lem:Close path endpoint compatibilty}. This proves that $F(z)$ does not depend on the path $\gamma$.\par
		It remains to show that $F$ is a local isometry. Indeed consider a point $z \in B_R(x)$ and a fat isometry caterpillar $f_t$ along some rectifiable path $\gamma:[a,b]\rightarrow M$ from $x$ to $z$. Then $F$ coincides with the isometry $f_b :B_{r'}(\gamma(b)) \rightarrow B_{r'}(f_b(\gamma(b)))$ on $B_{r'}(\gamma(b))$. To see this, consider a point $v \in B_{r'}(z)$, we have shown above that the point $F(v)$ does not depend on the path. Therefore consider the path $\gamma$ followed by a shortest path from $z$ to $v$, then $f_t$ defines a fat isometry caterpillar along that path. It follows that $f_b(v)=F(v)$.
	\end{proof}
	
	In order to prove \autoref{Prop: Extension Property}, it remains to show that the extension $F$ constructed in \autoref{lem:Existence of global extension} is an isometry.

	\begin{proof}[Proof of \autoref{Prop: Extension Property}] With the notation coming from \autoref{lem:Existence of global extension} we will show that for $R':=R/100$ the map $F\vert_{\overline{B}_{R'}(x)}$ is an isometry. This is clearly sufficient in order to prove \autoref{Prop: Extension Property}.\par
		Observe that since $F$ is a local isometry, it has the homotopy lifting property. The reason we consider $R'$ instead of $R$ is to make sure that the lifts of paths stay in $B_R(x)$ and thus are well defined. The homotopy lifting property implies that $F$ is injective. Indeed consider two points $p,q \in \overline{B}_{R'}(x)$ satisfying $F(p)=F(q)$, then consider the shortest paths $\gamma_1$ from $x$ to $p$ and $\gamma_2$ from $x$ to $q$. The paths $F(\gamma_1)$ and $F(\gamma_2)$ are homotopic, lifting this homotopy gives $p=q$.\par
		Now using injectivity obtained above, we show that $F\vert_{\overline{B}_{R'}(x)}$ is distance-preserving. Consider two distinct points $v,w \in F(\overline{B}_{R'}(x))$ and $\gamma$ an arbitrary shortest path between them. Its lift $\tilde{\gamma}$ has the same length as $\gamma$, since $F$ is a local isometry. Observe that the map $F$ is $1$-Lipschitz, hence $\tilde{\gamma}$ must be a shortest path. Therefore $F$ is distance preserving.\par
		It remains to show that $F\vert_{\overline{B}_{R'}(x)}$ is surjective. Assume this is not the case, then there is a $z\in \overline{B}_{R'}(y)$ not in the image of $F\vert_{\overline{B}_{R'}(x)}$. Set $\rho:=d(y,z)$ and consider the isometry $f_{yx}:\overline{B}_{\rho}(y) \rightarrow \overline{B}_{\rho}(x)$ existing by \autoref{Lower domain bound}. Then $F\vert_{\overline{B}_{R'}(x)} \circ f_{yx}:\overline{B}_{\rho}(y) \rightarrow\overline{B}_{\rho}(y) $ is a distance preserving map, mapping a compact subset to a proper subset of itself. This is a contradiction. Therefore $F\vert_{\overline{B}_{R'}(x)}$ is a bijective distance-preserving map and thus an isometry.
	\end{proof}
	
	In order to prove \autoref{Prop: Extension Property}, it remains to show that the extension $F$ constructed in \autoref{lem:Existence of global extension} is an isometry.

	\begin{proof}[Proof of \autoref{Prop: Extension Property}] With the notation coming from \autoref{lem:Existence of global extension} we will show that for $R':=R/100$ the map $F\vert_{\overline{B}_{R'}(x)}$ is an isometry. This is clearly sufficient in order to prove \autoref{Prop: Extension Property}.\par
		Observe that since $F$ is a local isometry, it has the homotopy lifting property. The reason we consider $R'$ instead of $R$ is to make sure that the lifts of paths stay in $B_R(x)$ and thus are well defined. The homotopy lifting property implies that $F$ is injective. Indeed consider two points $p,q \in \overline{B}_{R'}(x)$ satisfying $F(p)=F(q)$, then consider the shortest paths $\gamma_1$ from $x$ to $p$ and $\gamma_2$ from $x$ to $q$. The paths $F(\gamma_1)$ and $F(\gamma_2)$ are homotopic, lifting this homotopy gives $p=q$.\par
		Now using injectivity obtained above, we show that $F\vert_{\overline{B}_{R'}(x)}$ is distance-preserving. Consider two distinct points $v,w \in F(\overline{B}_{R'}(x))$ and $\gamma$ an arbitrary shortest path between them. Its lift $\tilde{\gamma}$ has the same length as $\gamma$, since $F$ is a local isometry. Observe that the map $F$ is $1$-Lipschitz, hence $\tilde{\gamma}$ must be a shortest path. Therefore $F$ is distance preserving.\par
		It remains to show that $F\vert_{\overline{B}_{R'}(x)}$ is surjective. Assume this is not the case, then there is a $z\in \overline{B}_{R'}(y)$ not in the image of $F\vert_{\overline{B}_{R'}(x)}$. Set $\rho:=d(y,z)$ and consider the isometry $f_{yx}:\overline{B}_{\rho}(y) \rightarrow \overline{B}_{\rho}(x)$ existing by \autoref{Lower domain bound}. Then $F\vert_{\overline{B}_{R'}(x)} \circ f_{yx}:\overline{B}_{\rho}(y) \rightarrow\overline{B}_{\rho}(y) $ is a distance preserving map, mapping a compact subset to a proper subset of itself. This is a contradiction. Therefore $F\vert_{\overline{B}_{R'}(x)}$ is a bijective distance-preserving map and thus an isometry.
	\end{proof}
	\section{Local Lie group structure}\label{subseq: local lie group structure}
	The goal of this section is to show, that an open subset of $(M,d_{g_0})$ is isometric to a smooth locally homogeneous space.\par
	In \autoref{subseq: extracting smooth structure} we define, using the local isometries and \autoref{Prop: Extension Property}, a local group $G$ acting transitively on an open subset of $O \subset M$ and prove that $G$ is locally isomorphic to a Lie group. This makes it possible to write $O$ as a local quotient of this Lie group by a local isotropy group.
	\par
	In \autoref{subseq: using smooth structure} we find
	left invariant Riemannian metric on this quotient which  turns the canonical homeomorphism  into an isometry.

	\subsection{Extracting the smooth structure}\label{subseq: extracting smooth structure}
	We start by defining the local group. \par Let us collect the data for the local group in the sense of \autoref{defi: local group}.
	Starting with our locally homogeneous $C^0$-Riemannnian manifold $M$ fix $r_0>0$ such that $B_{100 r_0}(x_0)$ is the ball coming from \autoref{Prop: Extension Property}. Denote by $F_1$ and $F_2$ the extensions of isometries $f_1,f_2:B_{r_0/10}(x_0) \rightarrow B_{r_0}(x_0)$, which exist due to \autoref{Prop: Extension Property}.
	\begin{enumerate}
		\item Endow the collection of maps $$ G:= \left \lbrace f:B_{\frac{r_0}{10}}(x_0) \rightarrow B_{r_0}(x_0) : f \text{ isometric, }  f(B_{\frac{r_0}{10}}(x_0)) \cap B_{\frac{r_0}{10}}(x_0) \neq \emptyset \right \rbrace.$$ with the compact open topology.
		\item On $\Omega:= \left \lbrace(f_2,f_1) \in G \times G :F_2 \circ f_1 \in G \right \rbrace$ define the partial multiplication by $m(f_2,f_1) := F_2 \circ f_1$. 
		\item The neutral element $e$ of $G$ is the identity map.
		\item For every $f \in G$ define the inversion operation: Given $f \in G$ there exists a point $ y \in f(B_{\frac{r_0}{10}}(x_0)) \cap B_{\frac{r_0}{10}}(x_0)$ and $r>0$ such that 
		$$B_r(y) \subset f(B_{\frac{r_0}{10}}(x_0)) \cap B_{\frac{r_0}{10}}(x_0).$$
		Consider the restriction of $f$ to $B_r(f^{-1}(y))$. Since $f$ is an isometry, its inverse map $f^{-1}: B_r(y) \rightarrow B_r{(f^{-1}(y))}$ is also an isometry and thus has an extension $F^{-1}:B_R(y) \rightarrow B_R(f{^{-1}}(y))$. Now define $i(f)$ to be the restriction  $F^{-1} \vert_{B_{\frac{r_0}{10}}(x_0)}$.
	\end{enumerate}
	To unburden the notation we will write $G$ and mean $G$ together with the partial group structure as specified above. With a slight abuse of notation, we will call $G$ the local isometries of $M$. \par 
	Now it is straightforward to verify that $G$ as defined above is indeed a local topological group, which is locally compact.
	\begin{prop}[Local isometries form a local group]\label{Definition of local group}
		Let $M$ be a locally homogeneous $C^0$-Riemannian manifold and denote by $G$ the local isometries, as defined above. Then $G$ is a locally compact local topological group and the canonical action $G \times B_{\frac{r_0}{10}}(x_0) \mapsto M ; (g,p) \mapsto g(p)$ is continuous.\par
	\end{prop}
	\begin{proof}
		The axioms (1)-(4) in \autoref{defi: local group} follow immediately from our construction. 
		Local compactness
		follows from an Arzela-Ascoli type argument.
	\end{proof}		
	We want to show that the local group $G$ of local isometries is locally isomorphic to a Lie group. In order to do this, we want to apply van den Dries-Goldbring globalization \cite{MR2743102} together with the Gleason-Yamabe theorem. An important observation is that small subgroups that could appear in $G$ are actually Lie groups. This fact is encoded in the next Lemma.
	\begin{lem}\label{lem: Small groups are Lie}
		Let $G$ be the local group defined in \autoref{Definition of local group}. Then there exists a neighborhood of the identity $U$, such that every locally compact subgroup $H$ satisfying $H \subset U$ is a Lie group.
	\end{lem}
	\begin{proof}
		Denote by $U \subset G$ the local isometries satisfying $\vert x_0 f(x_0) \vert < r_0/100$. Then for any locally compact subgroup $H \subset U$ one has
		$$ \sup_{f \in H} \vert x_0 f(x_0) \vert < \frac{r_0}{100}. $$
		This implies $H(B_{\frac{r_0}{100}}(x_0)) \subset B_{\frac{r_0}{50}}(x_0)$, meaning that for every $f \in H$ the restriction of $f$ to $\cup_{h \in H}  h(B_{\frac{r_0}{100}}(x_0))$ is defined. \par
		Observe that $\cup_{h \in H}  h(B_{\frac{r_0}{100}}(x_0))$ is an open set, which is invariant under all elements of $H$ since $H$ is a group. Moreover, $H$ is a group acting effectively via Lipschitz homeomorphisms on an open set of $\mathbb{R}^n$. So by the Hilbert-Smith theorem (\autoref{Thm: Hilbert smith}) it is a Lie group.
	\end{proof}
	\begin{prop}\label{prop:local Lie}
		The local group defined in \autoref{Definition of local group} is locally isomorphic to a Lie group.
	\end{prop} 
	\begin{proof}
		By \autoref{Definition of local group} the local group $G$ is locally compact, thus applying van den Dries-Goldbring globalization theorem \cite{MR2743102} produces a restriction $G \vert_U$ of $G$ and a topological group $\hat{G}$ such that $G \vert_U$ is a restriction of $\hat{G}$ (compare \autoref{def restriction of group}).\par
		Applying the Gleason-Yamabe theorem \cite{MR3237440}[Theorem 1.1.17] to an open neighborhood $U$ as in \autoref{lem: Small groups are Lie} yields an open subgroup $G'$ of $\hat{G}$ and a compact normal subgroup $K$ of $G'$ such that $G'/K$ is isomorphic to a Lie group. By \autoref{lem: Small groups are Lie} $K$ is a Lie group and therefore $G'$ is a Lie group as well by \cite{MR29910}[Theorem 1]. The claim now follows.
	\end{proof}
	By the above proposition we can assume that a local Lie group, which we denote by $U_G$, is acting on a subset of $M$ by local isometries. Consider the isotropy group $H$ of $G$, by the above some restriction $U_H$ of it is a local Lie group. One can define a local quotient $U_G/U_H$, such that $U_G/U_H$ is a smooth manifold, $U_G$ operates transitively on $U_G/U_H$ and it is homeomorphic to an open subset of $M$. Local factor spaces of this type have appeared in \cites{spiro1993remark,mostow1950extensibility} and a rigorous definition has been written by Pediconi \cite{Pediconi_2020}[Proposition 6.1].

	\subsection{Capitalizing the smooth structure}\label{subseq: using smooth structure}
	In this subsection we will use the smooth structure on $U_G/U_H$ obtained in the previous \autoref{subseq: extracting smooth structure}.
	Let us denote the canonical homeomorphism by
	$h:O  \rightarrow U_G/U_H $ where  $O \subset M$ and $U_G$ acts transitively on $O$.\par
	First  we find some left invariant Riemannian metric on the local quotient $M'=U_G/U_H $.
	In a second step we show: If we consider $M'$ together with this left invariant metric, then the canonical homeomorphism $h:O \rightarrow M'$ is a Lipshitz map (this statement is a modification of \cite{MR997464}[Lemma 1]).
	
	Then applying the Rademacher theorem we find a point $p$ such that $dh_p:T_pO \rightarrow T_{h(p)}M'$ is an isomorphism of tangent spaces. Using this isomorphism we can pushforward the continuous Riemannian metric of $M$ at a point and extend it to a left-invariant (and thus smooth) Riemannian metric on $M'$. Finally it will be shown, that this metric space is isometric to $(O,d_{g_0})$, which proves the main theorem.

	Our first intermediate goal is to construct a $U_G$-invariant Riemannian metric on $U_G/U_H$.
	If $G$ is a global Lie group and $H$ is a closed subgroup using an averaging procedure such a metric is known to exist if $G$ acts effectively on $G/H$ and the closure of $Ad_H$ is compact \cite{MR2394158}[Proposition 3.16].\par
	Returning to our setting there exists a global, connected, simply connected Lie group $\hat{G}$, which is locally isomorphic to $U_G$. Denote by $\mathfrak{g},\mathfrak{h}$ the Lie algebras of $U_G,U_H$ respectively and set $\hat{H}:= \langle\exp(\mathfrak{h}) \rangle$. In order to reproduce the averaging procedure mentioned above, it is sufficient to show the compactness of $Ad_{\hat{H}}$, this is carried out in the next Lemma.

	\begin{lem}
		\label{lem:reductive}
		Let $\hat{G},\hat{H}$ be as described above and denote by $\mathfrak{g},\mathfrak{h}$ their Lie algebras. Then $Ad_{\hat{H}}$ is compact. 
	\end{lem}
	
	\begin{proof}
		Observe that the stabilizer $H$ of $x_0$ is a compact global group and moreover by \autoref{Thm: Hilbert smith} is a Lie group. This means that the identity component $H_0$ is a compact, connected Lie group. Hence for any $h \in H_0$ there exist ${h}_1, \ldots , h_n \in U_H$ and $n \in \mathbb{N}$ such that $h={h}_1 \cdots {h}_n$. Similarly, if $i: U_G \mapsto \hat{G}$ denotes the local isomorphism and $\hat{U}_H:= i(U_H)$, then for every  $\hat{h} \in \hat{H}$ there are $\hat{h}_1, \ldots , \hat{h}_n \in \hat{U}_H$ and $n \in \mathbb{N}$ such that $\hat{h}=\hat{h}_1 \cdots \hat{h}_n$.\par
		Analogously to the proof of \autoref{lem: Small groups are Lie} one can define a neighborhood of the identity $U_0$ in $U_G$ satisfying $h U_0 h^{-1} \subset U_0 $ such that all products between elements in $U_0$ are defined. Using the invariance property of $U_0$ one can prove inductively for $g \in U_0$, $n \in \mathbb{N}$ and $h_1,\ldots, h_n \in U_0$ the identity $$ i ( h_1 \cdots h_n \cdot g  \cdot h_n^{-1} \cdots h_1^{-1})= i(h_1) \cdots i(h_n) \cdot i(g) \cdot i(h_n)^{-1} \cdots i(h_1)^{-1}.$$
		Observe that this property does not follow from the homomorphism property, since $n$-fold multiplication might not be defined.\par
		We have achieved the following: One can write conjugation with arbitrary elements in $H_0, \hat{H}$ in terms of elements of $U_H$ and $\hat{U}_H$. This makes it possible to relate $Ad_{H_0}$, which is known to be compact, to $Ad_{\hat{H}}$.\par
		For $h \in H_0$ write $h= {h}_1 \cdots {h}_n$ and define the map $$F_h:\hat{G} \rightarrow \hat{G}; g\mapsto  i({h}_1) \cdots i({h}_n) \cdot i(g) \cdot  i(h_n)^{-1} \cdots i({h}_1)^{-1}.$$
		By the formula above $F_h$ is well-defined. Consider the map $$f:H_0 \rightarrow Ad_{\hat{H}}, h \mapsto d_e(F_h).$$
		It remains to show that $f$ is continuous and surjective. Surjectivity follows from the formula above. For continuity of $f$ observe that it is sufficient to prove continuity in $e$. Then we have $f= Ad \circ i$, which is a composition of continuous maps. This finishes the proof.
	\end{proof}
	With \autoref{lem:reductive} and the remark preceding it we obtain:
	
	\begin{cor}
		\label{cor: Riemannian metric on quotient}
		There is a $U_G$-invariant Riemannian metric $\hat{g}$ on $U_G/U_H$.
	\end{cor}
	It follows from the next lemma, that the canonical homeomorphism $h:(M,d_{g_0}) \rightarrow (U_G/U_H,d_{\hat{g}})$ is a Lipschitz map.
	
	\begin{lem}[$f$ is Lipschitz]\label{lem: natural map lipschitz}
		Let $d^*$ be an intrinsic metric on $U_G/U_H$, which is $U_G$-invariant and let $g$ be a smooth $U_G$-invariant
		Riemannian metric  on $U_G/U_H$.
		Then for some constant $C>0$ we have $d_g\leq C d^*$. 
	\end{lem}
	\begin{proof}
		It is possible to apply almost the same 
		arguments as in  \cite{MR997464}[Lemma 1].
		The proof uses the fact that the
		exponential map for smooth metric is a diffeomorphism near the origin,  triangle inequality
		and transitivity of the group action.
	\end{proof}
	
	Denote by $\varphi_1 : \Omega_1 \subset \mathbb{R}^n \rightarrow O,\varphi_2:\Omega_2 \subset \mathbb{R}^n \rightarrow U_G/U_H $ charts around $x_0$, $eH$ respectively. In view of \autoref{lem: Maps are bilip}, \autoref{lem: natural map lipschitz} we have that the coordinate expression $F:=  \varphi_2^{-1} \circ f \circ  \varphi_1$ is a Lipschitz map between open subsets of $\mathbb{R}^n$. By the Rademacher theorem $F$ is differentiable almost everywhere and we have the area formula \cite{MR1158660}[3.3.2]
	$$ \int_{\Omega_1} \det(dF_x) dx = \int_{\mathbb{R}^n} \mathcal{H}^0(\Omega_1 \cap F^{-1}(z)) \, dz = \int_{F(\Omega_1)} 1 \, dz>0,$$
	where $\mathcal{H}^0$ denotes the counting measure. The second equality comes from the fact that $F$ is a homeomorphism.
	This implies that there is a point such that $dF$ is an isomorphism, without loss of generality at $\varphi_1^{-1}(x_0)$, then $df$ is an isomorphism at $x_0$. Define a scalar product $g^*$ at $eH$ by the formula 
	$$g^*_{eH}( v,w) := g_0
	( (df)^{-1}_{eH} v, (df)^{-1}_{eH}  w )(x_0).$$\par
	Observe that one can define a smooth Riemannian metric $g^*$ by 
	$$ g^*_{{gH}}(v'  , w' ):= g^*_{eH}( dL_{g^{-1}} v' , dL_{g^{-1}} w' ) .$$
	The expression above is well-defined, since $g^*_{eH}$ is 
	$Ad_{H}$-invariant. This comes from the fact that it is defined in terms of a linear isomorphism, whose original map is adapted to the group action.\par

	We are now ready to prove the main theorem
	\begin{proof}[Proof of Main Theorem]

		First we want to show that our construction implies that the map $h: (O,d_0)\rightarrow(U_G/U_H,d_{g*}) $ is an isometry.
		
		follows
		\autoref{Equality condtion for Riem.metrics}.\par
		We have now proven a local version of the main theorem. For the  general statement we construct a smooth chart for some neighbourhood
		of an arbitrary point using the local homogenity condition and
		the local statement above. Then we observe that 
		transition functions are smooth since
		by the Meyers-Steenrod theorem isometries between smooth manifolds are actually smooth maps.                   
	\end{proof}
	\begin{lem}[Equality condition for homogeneous Riemannian metrics]\label{Equality condtion for Riem.metrics}
		Let $(\Omega,g)$ $(\Omega',g')$
		be locally homogeneous  $C^0$-Riemannian manifolds
		and $h:\Omega\to \Omega'$ be a homeomorphism satisfying the following conditions
		\begin{enumerate}
			\item The map $h$ respects the local isometries of $\Omega$ in the following sense:
			For any $x,y\in \Omega$ there exist open neighborhoods $x \in U_x, y \in U_y$ and an isometry $f_{xy}:U_x\to U_y$ such that
			$h\circ f_{xy}\circ h^{-1}:h(U_x)\to h(U_y)$ is an isometry as well.
			\item There exists a point $p \in \Omega$ such that the map $h$ is differentiable at $p$ and $d_ph$ is an isometry.
		\end{enumerate}
		Then the map $h: (\Omega,d_{g})  \rightarrow (\Omega',d_{g'}) $ is an isometry, where $d_g,d_{g'}$ denotes the metrics induced by $g,g'$ respectively.
	\end{lem}
	
	\begin{proof} 
		It is sufficient to show that the length of rectifiable
		curves is the same with regard to both Riemannian metrics. That means $L_{d_g}(\gamma)=L_{d_{g'}}( h \circ \gamma)$ for every rectifiable curve $\gamma$ in $M$.\par
		
		Using the $C^0$-Riemannian manifold structure we can construct for every $\varepsilon>0$ an $(1+\varepsilon)$-bilipschitz chart around $p$ and $h(p)$ respectively. From this together with condition (2) we can establish for all $x$ in a neighborhood of $p$ the following estimate 
		\begin{equation}\label{eq:Equality condition}
		\vert d_g(p,x) - d_{g'}(h (p),h (x)) \vert < \varepsilon \cdot d_{g}(p,x) .
		\end{equation}
		More precisely: If $h$ is a map between two open subsets of $\mathbb{R}^n$, which is differentiable at a fixed point $p$, then the above estimate is obvious. Since our $(1+\varepsilon)$-bilipschitz charts almost not distort distances the estimate holds in the general situation as well.\par
		Denote by $\gamma:[0,1]\rightarrow \Omega$ a rectifiable curve and by $0=z_0 < z_1 < \ldots < z_N = 1$ a partition of $[0,1]$. We will show $$L(\gamma) \geq \sum_{i=0}^{N-1} d_{g'}(h \circ \gamma(z_i),h \circ \gamma(z_{i+1})),$$ which implies $L(\gamma)\geq L(h \circ \gamma)$. Indeed from assumption (1) of \autoref{Equality condtion for Riem.metrics} and inequality (\ref{eq:Equality condition}) we have for any $\varepsilon>0$ and any $0 \leq i \leq N-1$ the inequality $$d_{g'}(h \circ \gamma(z_i), h \circ \gamma(z_{i+1}) \leq (1+\varepsilon) \cdot d_g (\gamma(z_i), \gamma(z_{i+1})).$$
		Since $\varepsilon$ was arbitrary summation implies the claim.\par
		We have established that $h \circ \gamma$ is rectifiable, therefore by a similar argument as above we obtain $L(h) \geq L(h \circ \gamma) $, which yields the result.

	\end{proof}

	\bibliography{bibliography}

\begin{bibdiv}
\begin{biblist}

\bib{MR997464}{article}{
      author={Berestovski\u{\i}, V.~N.},
       title={Homogeneous manifolds with an intrinsic metric. {II}},
        date={1989},
        ISSN={0037-4474},
     journal={Sibirsk. Mat. Zh.},
      volume={30},
      number={2},
       pages={14\ndash 28, 225},
         url={https://doi.org/10.1007/BF00971372},
      review={\MR{997464}},
}

\bib{bohm2019optimal}{article}{
      author={B{\"o}hm, Christoph},
      author={Lafuente, Ramiro},
      author={Simon, Miles},
       title={Optimal curvature estimates for homogeneous ricci flows},
        date={2019},
     journal={International Mathematics Research Notices},
      volume={2019},
      number={14},
       pages={4431\ndash 4468},
}

\bib{MR0413144}{book}{
      author={Bredon, Glen~E.},
       title={Introduction to compact transformation groups},
   publisher={Academic Press, New York-London},
        date={1972},
        note={Pure and Applied Mathematics, Vol. 46},
      review={\MR{0413144}},
}

\bib{MR3358543}{article}{
      author={Burtscher, Annegret~Y.},
       title={Length structures on manifolds with continuous {R}iemannian
  metrics},
        date={2015},
     journal={New York J. Math.},
      volume={21},
       pages={273\ndash 296},
         url={http://nyjm.albany.edu:8000/j/2015/21_273.html},
      review={\MR{3358543}},
}

\bib{MR283727}{article}{
      author={Calabi, Eugenio},
      author={Hartman, Philip},
       title={On the smoothness of isometries},
        date={1970},
        ISSN={0012-7094},
     journal={Duke Math. J.},
      volume={37},
       pages={741\ndash 750},
         url={http://projecteuclid.org/euclid.dmj/1077379303},
      review={\MR{283727}},
}

\bib{MR2394158}{book}{
      author={Cheeger, Jeff},
      author={Ebin, David~G.},
       title={Comparison theorems in {R}iemannian geometry},
   publisher={AMS Chelsea Publishing, Providence, RI},
        date={2008},
        ISBN={978-0-8218-4417-5},
         url={https://doi.org/10.1090/chel/365},
        note={Revised reprint of the 1975 original},
      review={\MR{2394158}},
}

\bib{MR1109664}{article}{
      author={De~Cecco, Giuseppe},
      author={Palmieri, Giuliana},
       title={Integral distance on a {L}ipschitz {R}iemannian manifold},
        date={1991},
        ISSN={0025-5874},
     journal={Math. Z.},
      volume={207},
      number={2},
       pages={223\ndash 243},
         url={https://doi.org/10.1007/BF02571386},
      review={\MR{1109664}},
}

\bib{MR1318157}{article}{
      author={De~Cecco, Giuseppe},
      author={Palmieri, Giuliana},
       title={L{IP} manifolds: from metric to {F}inslerian structure},
        date={1995},
        ISSN={0025-5874},
     journal={Math. Z.},
      volume={218},
      number={2},
       pages={223\ndash 237},
         url={https://doi.org/10.1007/BF02571901},
      review={\MR{1318157}},
}

\bib{MR1158660}{book}{
      author={Evans, Lawrence~C.},
      author={Gariepy, Ronald~F.},
       title={Measure theory and fine properties of functions},
      series={Studies in Advanced Mathematics},
   publisher={CRC Press, Boca Raton, FL},
        date={1992},
        ISBN={0-8493-7157-0},
      review={\MR{1158660}},
}

\bib{MR1273464}{article}{
      author={Fukaya, Kenji},
      author={Yamaguchi, Takao},
       title={Isometry groups of singular spaces},
        date={1994},
        ISSN={0025-5874},
     journal={Math. Z.},
      volume={216},
      number={1},
       pages={31\ndash 44},
         url={https://doi.org/10.1007/BF02572307},
      review={\MR{1273464}},
}

\bib{MR29910}{article}{
      author={Gleason, Andrew~M.},
       title={On the structure of locally compact groups},
        date={1949},
        ISSN={0027-8424},
     journal={Proc. Nat. Acad. Sci. U.S.A.},
      volume={35},
       pages={384\ndash 386},
         url={https://doi.org/10.1073/pnas.35.7.384},
      review={\MR{29910}},
}

\bib{MR2680491}{article}{
      author={Goldbring, Isaac},
       title={Hilbert's fifth problem for local groups},
        date={2010},
        ISSN={0003-486X},
     journal={Ann. of Math. (2)},
      volume={172},
      number={2},
       pages={1269\ndash 1314},
         url={https://doi.org/10.4007/annals.2010.172.1273},
      review={\MR{2680491}},
}

\bib{MR3914958}{article}{
      author={Guijarro, Luis},
      author={Santos-Rodr\'{\i}guez, Jaime},
       title={On the isometry group of {$RCD^*(K,N)$}-spaces},
        date={2019},
        ISSN={0025-2611},
     journal={Manuscripta Math.},
      volume={158},
      number={3-4},
       pages={441\ndash 461},
         url={https://doi.org/10.1007/s00229-018-1010-7},
      review={\MR{3914958}},
}

\bib{MR2488667}{article}{
      author={Halverson, Denise~M.},
      author={Repov\v{s}, Du\v{s}an},
       title={The {B}ing-{B}orsuk and the {B}usemann conjectures},
        date={2008},
        ISSN={1331-0623},
     journal={Math. Commun.},
      volume={13},
      number={2},
       pages={163\ndash 184},
      review={\MR{2488667}},
}

\bib{MR38111}{article}{
      author={Hartman, Philip},
       title={On the local uniqueness of geodesics},
        date={1950},
        ISSN={0002-9327},
     journal={Amer. J. Math.},
      volume={72},
       pages={723\ndash 730},
         url={https://doi.org/10.2307/2372288},
      review={\MR{38111}},
}

\bib{MR1393941}{book}{
      author={Kobayashi, Shoshichi},
      author={Nomizu, Katsumi},
       title={Foundations of differential geometry. {V}ol. {II}},
      series={Wiley Classics Library},
   publisher={John Wiley \& Sons, Inc., New York},
        date={1996},
        ISBN={0-471-15732-5},
        note={Reprint of the 1969 original, A Wiley-Interscience Publication},
      review={\MR{1393941}},
}

\bib{MR3925112}{article}{
      author={Lytchak, Alexander},
      author={Nagano, Koichi},
       title={Geodesically complete spaces with an upper curvature bound},
        date={2019},
        ISSN={1016-443X},
     journal={Geom. Funct. Anal.},
      volume={29},
      number={1},
       pages={295\ndash 342},
         url={https://doi.org/10.1007/s00039-019-00483-7},
      review={\MR{3925112}},
}

\bib{MR2262730}{article}{
      author={Lytchak, Alexander},
      author={Schroeder, Viktor},
       title={Affine functions on {${\rm CAT}(\kappa)$}-spaces},
        date={2007},
        ISSN={0025-5874},
     journal={Math. Z.},
      volume={255},
      number={2},
       pages={231\ndash 244},
         url={https://doi.org/10.1007/s00209-006-0020-4},
      review={\MR{2262730}},
}

\bib{MR2216252}{article}{
      author={Lytchak, Alexander},
      author={Yaman, Asli},
       title={On {H}\"{o}lder continuous {R}iemannian and {F}insler metrics},
        date={2006},
        ISSN={0002-9947},
     journal={Trans. Amer. Math. Soc.},
      volume={358},
      number={7},
       pages={2917\ndash 2926},
         url={https://doi.org/10.1090/S0002-9947-06-04195-X},
      review={\MR{2216252}},
}

\bib{mostow1950extensibility}{article}{
      author={Mostow, George~Daniel},
       title={The extensibility of local lie groups of transformations and
  groups on surfaces},
        date={1950},
     journal={Annals of mathematics},
       pages={606\ndash 636},
}

\bib{MR3728284}{book}{
      author={Munkres, James~R.},
       title={Topology},
   publisher={Prentice Hall, Inc., Upper Saddle River, NJ},
        date={2000},
        ISBN={0-13-181629-2},
      review={\MR{3728284}},
}

\bib{MR1503467}{article}{
      author={Myers, S.~B.},
      author={Steenrod, N.~E.},
       title={The group of isometries of a {R}iemannian manifold},
        date={1939},
        ISSN={0003-486X},
     journal={Ann. of Math. (2)},
      volume={40},
      number={2},
       pages={400\ndash 416},
         url={https://doi.org/10.2307/1968928},
      review={\MR{1503467}},
}

\bib{MR1088511}{article}{
      author={Nicolodi, L.},
      author={Tricerri, F.},
       title={On two theorems of {I}. {M}. {S}inger about homogeneous spaces},
        date={1990},
        ISSN={0232-704X},
     journal={Ann. Global Anal. Geom.},
      volume={8},
      number={2},
       pages={193\ndash 209},
         url={https://doi.org/10.1007/BF00128003},
      review={\MR{1088511}},
}

\bib{MR1274133}{article}{
      author={Otsu, Yukio},
      author={Shioya, Takashi},
       title={The {R}iemannian structure of {A}lexandrov spaces},
        date={1994},
        ISSN={0022-040X},
     journal={J. Differential Geom.},
      volume={39},
      number={3},
       pages={629\ndash 658},
         url={http://projecteuclid.org/euclid.jdg/1214455075},
      review={\MR{1274133}},
}

\bib{pediconi2019compactness}{misc}{
      author={Pediconi, Francesco},
       title={A compactness theorem for locally homogeneous spaces},
        date={2019},
}

\bib{Pediconi_2020}{article}{
      author={Pediconi, Francesco},
       title={A local version of the myers–steenrod theorem},
        date={2020Jun},
        ISSN={1469-2120},
     journal={Bulletin of the London Mathematical Society},
         url={http://dx.doi.org/10.1112/blms.12367},
}

\bib{MR1464908}{article}{
      author={Repov\v{s}, Du\v{s}an},
      author={\v{S}\v{c}epin, Evgenij},
       title={A proof of the {H}ilbert-{S}mith conjecture for actions by
  {L}ipschitz maps},
        date={1997},
        ISSN={0025-5831},
     journal={Math. Ann.},
      volume={308},
      number={2},
       pages={361\ndash 364},
         url={https://doi.org/10.1007/s002080050080},
      review={\MR{1464908}},
}

\bib{MR131248}{article}{
      author={Singer, I.~M.},
       title={Infinitesimally homogeneous spaces},
        date={1960},
        ISSN={0010-3640},
     journal={Comm. Pure Appl. Math.},
      volume={13},
       pages={685\ndash 697},
         url={https://doi.org/10.1002/cpa.3160130408},
      review={\MR{131248}},
}

\bib{spiro1992lie}{article}{
      author={Spiro, A},
       title={Lie pseudogroups and locally homogeneous riemannian spaces},
        date={1992},
     journal={BOLLETTINO DELLA UNIONE MATEMATICA ITALIANA},
      volume={6},
      number={4},
       pages={843\ndash 872},
}

\bib{spiro1993remark}{article}{
      author={Spiro, Andrea},
       title={A remark on locally homogeneous riemannian spaces},
        date={1993},
     journal={Results in Mathematics},
      volume={24},
      number={3-4},
       pages={318\ndash 325},
}

\bib{MR3237440}{book}{
      author={Tao, Terence},
       title={Hilbert's fifth problem and related topics},
      series={Graduate Studies in Mathematics},
   publisher={American Mathematical Society, Providence, RI},
        date={2014},
      volume={153},
        ISBN={978-1-4704-1564-8},
      review={\MR{3237440}},
}

\bib{MR2743102}{article}{
      author={van~den Dries, Lou},
      author={Goldbring, Isaac},
       title={Globalizing locally compact local groups},
        date={2010},
        ISSN={0949-5932},
     journal={J. Lie Theory},
      volume={20},
      number={3},
       pages={519\ndash 524},
      review={\MR{2743102}},
}

\end{biblist}
\end{bibdiv}
\end{document}